\documentclass[11pt,letterpaper,abstracton,final]{scrartcl}

\usepackage{hyperref}

\usepackage[canadian]{babel}
\usepackage[utf8]{inputenc}

\makeatletter
\def\aliasbib#1#2{\expandafter\xdef
\csname b@#1\endcsname{\csname b@#2\endcsname}}
\makeatother

\aliasbib{pultr}{Pul:The-right-adjoints}
\aliasbib{nestar1}{NesTar:Dual}
\aliasbib{HN:Homomorphisms}{HelNes:GrH}

\usepackage{amsthm}
\newtheorem{lemma}{Lemma}[section]
\newtheorem{theorem}[lemma]{Theorem}

\newtheorem{prop}[lemma]{Proposition}

\newtheorem{claim}{Claim}

\theoremstyle{definition}
\newtheorem{defi}[lemma]{Definition}

\usepackage{amsmath}
\usepackage{pxfonts}

\usepackage{amssymb}

\usepackage[pdftex]{graphicx}
\usepackage{paralist}

\let\Rho\Omega
\def\S{\mathcal{S}}
\def\T{\mathcal{T}}
\def\F{\mathcal{F}}
\let\notto\nrightarrow
\def\bul{^{\bullet}}

\title{Digraph functors which admit both~left~and~right adjoints}
\author{Jan Foniok%
\thanks{Department of Mathematics and Statistics,
Queen's University,
Jeffery Hall,
48 University Avenue,
Kingston ON K7L 3N6,
Canada.
Fully supported by the Swiss National Science Foundation.
Current address:
School of Computing, Mathematics and Digital Technology,
Manchester Metropolitan University,
Chester Street,
Manchester M1~5GD,
England;
j.foniok@mmu.ac.uk.}\and
Claude Tardif%
\thanks{Department of Mathematics and Computer Science,
Royal Military College of Canada,
PO Box 17000, Station ``Forces'',
Kingston ON K7K 7B4,
Canada;
Claude.Tardif@rmc.ca.
Supported by grants from NSERC and ARP.}}

\begin{document}

%%%%%%%  abstract %%%%%%%%%%%%

\date{12 August 2014}
\maketitle

\begin{abstract}
For our purposes, two functors $\Lambda$ and $\Gamma$ are said to be 
adjoint if for any digraphs
$G$ and~$H$, there exists a homomorphism of~$\Lambda(G)$ to~$H$
if and only if there exists a homomorphism of~$G$ to~$\Gamma(H)$.
We investigate the right adjoints characterised by Pultr in
[A.~Pultr, The right adjoints into the categories of relational systems,
 In {\em Reports of the Midwest Category Seminar, IV}, volume 137 of
{\em Lecture Notes in Mathematics}, pages 100--113, Berlin, 1970].
We find necessary conditions for these functors to admit right adjoints themselves.
We give many examples where these necessary conditions are satisfied,
and the right adjoint indeed exists. Finally, we discuss a connection between
these right adjoints and homomorphism dualities.

\end{abstract}

\section{Introduction}

A {\em digraph functor} is a construction~$\Gamma$ which makes a
digraph~$\Gamma(G)$ out of a digraph~$G$, such that if there exists
a homomorphism of~$G$ to~$H$, then there exists a homomorphism
of~$\Gamma(G)$ to~$\Gamma(H)$.
Consider two functors $\Lambda$ and~$\Gamma$.
We say that $\Lambda$~is a {\em left adjoint} of~$\Gamma$ and
$\Gamma$~is a {\em right adjoint} of~$\Lambda$ if the existence
of a homomorphism of~$\Lambda(G)$ to~$H$  is equivalent to the
existence of a homomorphism of~$G$ to~$\Gamma(H)$.
Note that the precise categorial definition requires a natural
correspondence between the morphisms of~$\Lambda(G)$ to~$H$ and
those of~$G$ to~$\Gamma(H)$, but for our applications it is usually
enough to distinguish between the existence and non-existence of a
homomorphism between two digraphs.
Hence we work in the ``thin'' category of digraphs, in which there
is at most one generic morphism from one digraph to another.

Some significant constructions in graph theory turn out to be functors, and 
sometimes some of their fundamental properties are related to the fact that
they have a left adjoint, a right adjoint, and sometimes both. Thus it is 
worth the while to characterise the pairs of adjoint functors. However, that
objective may be out of reach for the moment. Our purpose is to lay groundwork
in that direction.

In a sense, Pultr~\cite{Pul:The-right-adjoints} has already characterised the
pairs of adjoint functors. Nevertheless, his characterisation holds in the 
category of multidigraphs, where morphisms must specify images of vertices 
and of arcs. Right-left adjunction is preserved in the thin category 
of digraphs; however, more pairs of adjoint functors exist. In particular,
some of the right adjoint functors of Pultr themselves admit right adjoints,
while no such adjoints exist in the category of multidigraphs.

Some of the adjoints characterised by Pultr are in fact well-known
constructions in graph theory.  The left adjoints of Pultr encompass
arc subdivisions and standard products, and his right adjoints
encompass the shift graph construction and exponentiation.

We will call the left and right functors characterised by Pultr respectively
{\em left Pultr functors} and {\em central Pultr functors}. One feasible
objective seems to be the characterisation of the central Pultr functors
which admit right adjoints. In Theorem~\ref{thm:tree}, we prove necessary
conditions. In Sections~\ref{exa:arc} to~\ref{sec:tree}, we show that some of the functors
satisfying these conditions do indeed admit right adjoints, which leads us to wonder
whether these conditions are also sufficient.

The existence of central Pultr functors with right adjoints allows us to define
new pairs of adjoint functors by composition. It is not at all clear whether
compositions of Pultr functors suffices to define all pairs of digraph adjoint 
functors.

\section{Pultr templates and functors} \label{ptaf}

A {\em homomorphism} is an arc-preserving map between digraphs.
If $G,H$ are digraphs, we write $G\to H$ if there exists a homomorphism
of~$G$ to~$H$. $G$ and $H$ are called {\em homomorphically equivalent}
if $G\to H$ and $H\to G$. 

\begin{defi} A {\em Pultr template} is a quadruple $\mathcal{T} = (P,Q,\epsilon_1,\epsilon_2)$
where $P$, $Q$ are digraphs and $\epsilon_1, \epsilon_2$ homomorphisms of $P$ to $Q$.
\end{defi}

\begin{defi}
\label{dfn:central}
Given a Pultr template $\mathcal{T} = (P,Q,\epsilon_1,\epsilon_2)$
the {\em central Pultr functor} $\Gamma_{\mathcal{T}}$ is the following construction:
For a digraph $H$, the vertices of $\Gamma_{\mathcal{T}}(H)$ are the homomorphisms
$g : P \rightarrow H$, and the arcs of $\Gamma_{\mathcal{T}}(H)$ are the couples
$(g_1,g_2)$ such that there exists a homomorphism $h : Q \rightarrow H$ with
$g_1 = h \circ \epsilon_1$, $g_2 = h \circ \epsilon_2$. 
\end{defi}

Where $P$ is a set or a digraph and $\epsilon$~is a mapping,
by $\epsilon[P]$ we mean the image of this set or digraph
(where the mapping is applied elemet-wise).

\begin{defi} Given a Pultr template, $\mathcal{T} = (P,Q,\epsilon_1,\epsilon_2)$
the {\em left Pultr functor} $\Lambda_{\mathcal{T}}$ is the following construction:
For a digraph $G$, $\Lambda_{\mathcal{T}}(G)$ contains one copy $P_u$ of $P$
for every vertex $u$ of $G$, and for every arc $(u,v)$ of $G$, 
$\Lambda_{\mathcal{T}}(G)$ contains a copy $Q_{u,v}$ of $Q$ with $\epsilon_1[P]$ identified
with $P_u$ and $\epsilon_2[P]$ identified with $P_v$.
\end{defi}

Formally, the identification of $P$ and~$P'$ is described by defining
an equivalence relation in which the corresponding vertices of~$P$
and~$P'$ are equivalent, and then quotienting the digraph by this
relation; see~\cite{FonTar:Adjoint} for details.

\begin{theorem}[Pultr \cite{pultr}] \label{puthm}
For any Pultr template $\mathcal{T}$, $\Lambda_{\mathcal{T}}$ and
$\Gamma_{\mathcal{T}}$ are left and right adjoints.
\end{theorem}

For some templates ${\mathcal{T}}$, the central
Pultr functor $\Gamma_{\mathcal{T}}$ not only admits the left adjoint $\Lambda_{\mathcal{T}}$,
but also a right adjoint $\Omega_{\mathcal{T}}$. Not all templates have this
property. The following result provides necessary conditions.

\begin{theorem}
\label{thm:tree}
Let $\mathcal{T} = (P,Q,\epsilon_1,\epsilon_2)$ be a Pultr template such that 
$\Gamma_{\mathcal{T}}$ admits a right adjoint $\Rho_{\mathcal{T}}$.
Then $P$ and $Q$ are homomorphically equivalent to trees.
Moreover, for any tree $T$, $\Lambda_{\mathcal{T}}(T)$ is homomorphically 
equivalent to a tree.
\end{theorem}
\begin{proof}

Suppose that $\Gamma_{\mathcal{T}}$ admits a right adjoint $\Rho_{\mathcal{T}}$.
Let $H = \vec P_0$, the one-vertex digraph with no arcs.
Then a graph $G$ satisfies $\Gamma_{\mathcal{T}}(G) \rightarrow H$ if and only if
$\Gamma_{\mathcal{T}}(G)$ has no arcs, that is, $Q \not \rightarrow G$.
On the other hand, $\Gamma_{\mathcal{T}}(G) \rightarrow H$ if and only if
$G \rightarrow \Omega_{\mathcal{T}}(H)$, thus we have
$$
Q \not \rightarrow G \mbox{ if and only if } G \rightarrow \Omega_{\mathcal{T}}(H).
$$
Therefore $(Q,\Omega_{\mathcal{T}}(H))$ is a homomorphism duality
pair in the sense of~\cite{nestar1}.  It is known (by~\cite{Kom:Phd},
see also~\cite{nestar1}) that a digraph $Q$ is the left-hand member
of such a duality pair if and only if it is homomorphically equivalent
to a tree.

A similar argument (with $H = \emptyset$) shows that $P$ is also homomorphically
equivalent to a tree. More generally, let $T$ be a tree with dual $D(T)$.
Then for any digraph $G$,
\[
T \not \rightarrow \Gamma_{\mathcal{T}}(G) \ \Leftrightarrow \  \Gamma_{\mathcal{T}}(G) \rightarrow D(T),
\]
that is,
\[
\Lambda_{\mathcal{T}}(T) \not \rightarrow G \ \Leftrightarrow \ G \rightarrow \Rho_\mathcal{T}(D(T))
\]
Therefore $(\Lambda_{\mathcal{T}}(T), \Rho_\mathcal{T}(D(T)))$ is a duality pair,
whence $\Lambda_{\mathcal{T}}(T)$ is homomorphically equivalent to a tree.
\end{proof} 

It is not clear how to characterise the templates~$\T$ with the property
that for any tree~$T$, the left adjoint $\Lambda_\T(T)$ is homomorphically equivalent to a tree.
Small examples seem to suggest that templates with this property
have a $Q$ that is itself a tree, unless it is disconnected.
Also, it remains open whether the converse to Theorem~\ref{thm:tree}
holds. In the rest of this paper we establish partial results
in this respect. Our main result is Theorem~\ref{thm:tap} in Section~\ref{sec:tree},
which proves that for templates~$\T$ where $P$ is a vertex or an arc, and $Q$ is a tree,
$\Gamma_{\T}$ has a right adjoint $\Omega_{\T}$.
To facilitate the understanding of this main theorem without resorting to extremely formal definitions,
we begin by several examples.

\section{Example: The arc graph construction} \label{exa:arc}

We write $V(G)$ for the vertex set and $A(G)$ for the arc set of a
digraph~$G$.  If $x, y$ are vertices of~$G$, we sometimes write
$x\to y$ for $(x,y) \in A(G)$ (when there is no confusion about~$G$).
Note that we also write $G \rightarrow H$ for
``there exists a homomorphism of~$G$ to~$H$'',
but this notation is consistent, since the thin category of digraphs 
is itself a digraph. If
$X,Y$ are sets of vertices, we write $X\Rrightarrow Y$ if $x\to y$
for any $x\in X$ and any $y\in Y$. We abbreviate $X\Rrightarrow\{y\}$
to $X\Rrightarrow y$ and $\{x\}\Rrightarrow Y$ to $x\Rrightarrow Y$. 
Note that for any set $X$, we have
$\emptyset\Rrightarrow X \Rrightarrow \emptyset$.

The {\em arc graph} of a digraph $G$ is the digraph $\delta(G)$
constructed as follows: For every arc $x \rightarrow y$ of $G$,
$\delta(G)$ contains the vertex $(x,y)$, and for every pair
of consecutive arcs $x \rightarrow y \rightarrow z$ of $G$,
$\delta(G)$ contains the arc $(x,y) \rightarrow (y,z)$.
The arc graph construction is a well-known method for constructing
graphs with large odd girth and large chromatic number 
(see~\cite{HN:Homomorphisms}).

It turns out that the arc graph construction is a central Pultr functor.
We have $\delta(G) = \Gamma_\T(G)$ where $\T=(P,Q,\epsilon_1,\epsilon_2)$ with
$P=\vec P_1=(\{0,1\},\{(0,1)\})$, $Q=\vec
P_2=(\{0,1,2\},\{(0,1),(1,2)\})$, and $\epsilon_1$, $\epsilon_2$
mapping $P$ to the first and second arc of $Q$. The right adjoint $\Omega_{\T}$ 
of $\Gamma_{\T}=\delta$ exists; we call it $\delta_{R}$.

\begin{defi} 
For a digraph $H$, the vertices of $\delta_R(H)$ are all the pairs 
$(R^-,R^+)$ such that $R^-,R^+ \subseteq V(H)$ and $R^-\Rrightarrow R^+$. 
(Note that $(\emptyset,V(H))$ and $(V(H),\emptyset)$
are vertices of $\delta_R(H)$.) $(R^-,R^+)\to(S^-,S^+)$ is
an arc of $\delta_R(H)$ if and only if $R^+\cap S^-\ne\emptyset$.
\end{defi}

\begin{prop} For any digraphs $G$, $H$,
$\delta(G) \rightarrow H$ if and only if
$G \rightarrow \delta_R(H)$.
\end{prop}

\begin{proof}
Let $g:G\to\delta_R(H)$ be a homomorphism, with
$g(u) = (g_-(u),g_+(u))$. For any arc $u \to v$ of $G$,
we have $g(u) \to g(v)$, hence $g_+(u)\cap g_-(v) \neq \emptyset$.
Define $f:\delta(G)\to H$ by taking
$f(u,v)$ to be any element of $g_+(u)\cap g_-(v)$. 
Whenever $(x,y)\to(y,z)$
in~$\delta(G)$, we have $f(x,y)\in g_-(y) \Rrightarrow g_+(y) \ni f(y,z)$; hence
$f(x,y)\to f(y,z)$ in~$H$. Therefore $f$ is a homomorphism.

Conversely, let $f:\delta(G)\to H$ be a homomorphism.
For $u\in V(G)$, put
\begin{align*}
g_-(u) &= \{ f(x,u) \colon (x,u) \in A(G)\}, \\
g_+(u) &= \{ f(u,y) \colon (u,y) \in A(G)\}.
\end{align*}
Whenever $x \to u \to y$ in $G$, we have $(x,u)\to (u,y)$
in~$\delta(G)$, so $f(x,u)\to f(u,y)$ in~$H$ because $f$~is a
homomorphism. Thus $g_-(u)\Rrightarrow g_+(u)$ for any vertex~$u$
of~$G$. Hence $g(u) := (g_-(u),g_+(u))$~is a vertex of~$\delta_R(H)$.  Furthermore,
if $u\to v$ in~$G$, then $f(u,v)\in g_+(u)\cap g_-(v)$, and so
$g_+(u)\cap g_-(v)\ne\emptyset$. This shows that $g(u)\to g(v)$
in~$\delta_R(H)$. Therefore, $g: G \rightarrow \delta_R(H)$ defined
by $g(u) =  (g_-(u),g_+(u))$ is a homomorphism.
\end{proof}

Now, consider the template $\T=(P,Q,\epsilon_1,\epsilon_2)$ with
$P=\vec P_1=\bigl(\{0,1\},\{(0,1)\}\bigr)$, 
$Q=\bigl(\{0,1,2,3\},\{(0,1),(2,3),(0,2),(1,3)\}\bigr)$, and 
$\epsilon_1(0) = 0, \epsilon_1(1) = 1$, $\epsilon_2(0) = 2, \epsilon_2(1) = 3$.
For any graph $G$, $\delta(G)$ is a subgraph of $\Gamma_{\T}(G)$,
and $\Gamma_{\T}(G) = \delta(G)$ whenever $G$ does not contain
a $4$-cycle $x \rightarrow y \rightarrow z \leftarrow w \leftarrow x$. 
However, $\Gamma_{\T}$ does not have a right adjoint.
Indeed, even though $Q$ is homomorphically equivalent to a tree,
it is easy to see that for the tree 
$T = \bigl(\{0, \ldots, 5\},\{(0,1),(1,2),(3,2),(3,4),(4,5)\}\bigr)$,
$\Lambda_{\T}(T)$ is not homomorphically equivalent to a tree.
Thus by Theorem~\ref{thm:tree}, $\Gamma_{\T}$ does not have a right adjoint.

\section{Example: A path template}
\label{sec:path}

In \cite{HaT:Graph-powers,Tar:Mul}, (undirected) graph constructions 
are studied, which turn out to be right adjoints of central Pultr functors 
for templates $(P,Q,\epsilon_1,\epsilon_2)$ where $P$ is a point,
$Q$ is a path of odd length and $\epsilon_1,\epsilon_2$ map $P$
to the endpoints of $Q$. Similar constructions work for directed graphs,
we outline one example.

Let $Q$ be the oriented path $0 \leftarrow 1 \to 2 \to 3$. Let $P = \vec P_0 = (\{0\},\emptyset)$,
and $\epsilon_1(0)=0$, $\epsilon_2(0)=3$. For $\T = (P,Q,\epsilon_1,\epsilon_2)$,
a right adjoint $\Omega_{\T}$ of $\Gamma_{\T}$ is constructed as follows.

\begin{defi} 
\label{defi:path}
For a digraph $H$, the vertices of $\Omega_{\T}(H)$ are all the pairs 
$(a,A)$ such that $a \in V(H)$ and $A \subseteq V(H)$.
(Note that each $(a,\emptyset)$ with $a \in V(H)$ is a vertex of~$\Omega_{\T}(H)$.) 
$(a,A)\to(b,B)$ is an arc of $\Omega_{\T}(H)$ if and only if $b \in A \Rrightarrow B$.
\end{defi}

Note that for all $a, b \in V(H)$, $(a,\emptyset)$ is a sink, and $(b,B)$ is a source 
unless $b \Rrightarrow B$.

\begin{prop} \label{prop:rad3p} For any digraphs $G$, $H$,
$\Gamma_{\T}(G) \rightarrow H$ if and only if
$G \rightarrow \Omega_{\T}(H)$.
\end{prop}

\begin{proof}
Let $g:G\to\Omega_{\T}(H)$ be a homomorphism, with
$g(u) = (g_0(u),g_+(u))$. 
Define $f:\Gamma_{\T}(G)\to H$ by $f(u) = g_0(u)$. 
Whenever $u \to v$ in~$\Gamma_{\T}(G)$, 
there exist vertices $x, y$ such that $u \leftarrow x \to y \to v$ in $G$.
Since  $g$ is a homomorphism, we then have $g(u) \leftarrow g(x) \to g(y) \to g(v)$ 
in $\Omega_{\T}(H)$. By definition of adjacency in $\Omega_{\T}(H)$, this implies
$$
f(u) =  g_0(u) \in g_+(x) \Rrightarrow g_+(y) \ni g_0(v) = f(v).
$$
Therefore  $f(u)\to f(v)$, so $f$ is a homomorphism.

Conversely, let $f:\Gamma_{\T}(G)\to H$ be a homomorphism.
For $u\in V(G)$, put
$$
g_+(u) =  \{ f(y) \colon (u,y) \in A(G) \}.
$$
We define $g:G\to\Omega_{\T}(H)$ by $g(u) = (f(u),g_+(u))$.
If $(u,v)$ is an arc of $G$, then $f(v) \in g_+(u)$, and for every
$(u,x), (v,y) \in A(G)$, we have $x \leftarrow u \to v \to y$
in $G$, hence $x \to y$ in $\Gamma_{\T}(G)$ and $f(x) \to f(y)$ in $H$.
Thus $g_+(u) \Rrightarrow g_+(v)$. 
Therefore $(g(u),g(v))$ is an arc of $\Omega_{\T}(H)$. 
This shows that $g$ is a homomorphism.
\end{proof}

There are often many constructions of a right adjoint of a central Pultr functor,
even though they are all homomorphically equivalent.
The above construction of $\Omega_{\T}(H)$ is more compact than the construction
derived later from the proof of Theorem~\ref{thm:tap}, using one less coordinate. For future
reference, we provide a second construction in the spirit of the proof of Theorem~\ref{thm:tap}.

\begin{defi}
For a digraph $H$, the vertices of $\Omega_{\T}'(H)$ are all the triples 
$(a,A_1,A_2)$ such that $a \in V(H)$, $A_1, A_2 \subseteq V(H)$ and $A_1 \Rrightarrow A_2$.
$(a,A_1,A_2)\to(b,B_1,B_2)$ is an arc of $\Omega_{\T}'(H)$ if and only if 
$b \in A_1$ and $B_1 \subseteq A_2$. 
\end{defi}
The following can be proved in a similar fashion to Proposition~\ref{prop:rad3p}.
\begin{prop} For any digraphs $G$, $H$,
$\Gamma_{\T}(G) \rightarrow H$ if and only if
$G \rightarrow \Omega_{\T}'(H)$.
\end{prop}

We finish this section with a construction that gives the right
adjoint~$\Omega_\T$ for a fairly general family of templates~$\T$,
in which the respective $Q$'s are oriented paths.

\begin{defi}
\label{def:r-path}
Suppose that $Q$ is an oriented path with $m$ arcs and vertex set
$\{0,1,\dotsc,m\}$.
Let $\mathcal{T}=(\vec P_0,Q,\epsilon_1,\epsilon_2)$, where $\epsilon_1,\epsilon_2$ map the one-vertex graph~$\vec P_0$ to the
end-points of~$Q$. For a digraph~$H$, let the vertex set of~$\Rho_\T(H)$
be $V=\{(x,X_1,\dotsc,X_{m}) \colon x\in V(H),\ X_i\subseteq V(H)\text{
for each $i$, $X_m \Rrightarrow x$}\}$. There is an arc in~$\Rho_\T(H)$ from
$(x,X_1,\dotsc,X_{m})$ to $(y,Y_1,\dotsc, Y_{m})$ if
\begin{itemize}
\item[(1a)] $x\in Y_1$ if $0\to 1$ in $Q$,
\item[(1b)] $y\in X_1$ if $1\to 0$ in $Q$; and
\item[(2)] for each $i=1,\dotsc,m-1$:
\begin{itemize}
\item[(a)] $X_i\subseteq Y_{i+1}$ if $i\to i+1$ in $Q$,
\item[(b)] $Y_i\subseteq X_{i+1}$ if $i+1\to i$ in $Q$.
\end{itemize}
\end{itemize}
\end{defi}

The correctness of this construction, asserted in the following
proposition, follows from the proof of Theorem~\ref{thm:tap}.

\begin{prop}
Let $\T$ and $\Omega_\T$ be as in Definition~\ref{def:r-path}.
Then for any digraphs $G$, $H$, we have
$\Gamma_{\T}(G) \rightarrow H$ if and only if
$G \rightarrow \Omega_{\T}(H)$.
\end{prop}

\section{Examples: Compositions of adjoint functors and multiply exponential constructions}
\label{sec:compo}
The $k$-th iterated arc graph construction $\delta^k$ can be defined
recursively by $\delta^{k+1} = {\delta \circ \delta^k}$. It can also
be defined directly as $\delta^k = \Gamma_{\T}$, where 
$\T = (\vec P_k,\vec P_{k+1},\epsilon_1, \epsilon_2)$, with
\begin{gather*}
\begin{aligned}
\vec P_k &= \bigl(\{0,1,\ldots,k\},\ \{(0,1),(1,2),\ldots,(k-1,k)\}\bigr), \\
\vec P_{k+1} &= \bigl(\{0,1,\ldots,k+1\},\ \{(0,1),(1,2),\ldots,(k,k+1)\}\bigr),
\end{aligned} \\
\epsilon_1(i) = i \quad\text{and}\quad \epsilon_2(i) = i+1.
\end{gather*}
All iterated arc graph constructions admit right adjoints, defined
recursively by $\delta_R^{k+1} = \delta_R \circ \delta_R^k$.
In particular this shows the existence of a right Pultr adjoint
for templates with arbitrarily large~$P$. 

Note that the construction $\delta_R^k$ is an exponential construction iterated
$k$ times. In this case it turns out that multiply exponential size 
is necessary: For arbitrarily large integers~$n$, there are
graphs~$G$ such that $\chi(G) = n$ and $\chi(\delta^k(G)) = \Theta(\log^{(k)}(n))$.
For $m =   \Theta(\log^{(k)}(n))$, we then have $G \rightarrow \Omega(K_m)$,
for any right adjoint~$\Omega$ of~$\delta^k$. This means that $\Omega(K_m)$~needs at least
$n$ vertices. 

For any two central Pultr functors $\Gamma_1$, $\Gamma_2$, the composition $\Gamma_1 \circ \Gamma_2$
is a central Pultr functor for a suitably defined template. If $\Omega_1$, $\Omega_2$
are right adjoints of $\Gamma_1$ and $\Gamma_2$ respectively, then $\Gamma_1 \circ \Gamma_2$
admits a right adjoint, namely $\Omega_2 \circ \Omega_1$. For $\Gamma_1 = \delta$ and 
$\Gamma_2 = \Gamma_{\T}$, with $\T = (P,Q,\epsilon_1,\epsilon_2)$ being the template
of Section~\ref{sec:path}, we get $\Gamma_1 \circ \Gamma_2 = \Gamma_{\mathcal{U}}$,
with ${\mathcal{U}} = (\Lambda_{\T}(\vec{P}_1),\Lambda_{\T}(\vec{P}_2),\epsilon_1,\epsilon_2)$,
where $\epsilon_1,\epsilon_2$ map $\Lambda_{\T}(\vec{P}_1)$ respectively to the first
and second copy of $\Lambda_{\T}(\vec{P}_1)$ in $\Lambda_{\T}(\vec{P}_2)$.
It is not clear whether the doubly exponential construction of a right adjoint
$\Omega_{\T} \circ \delta_R$ of $\delta \circ \Gamma_{\T}$ is necessary in this case.
Composing in the reverse order, we get $\Gamma_{\T} \circ \delta = \Gamma_{\mathcal{V}}$,
where $\mathcal{V} = (\vec{P}_1, T, \epsilon_1, \epsilon_2)$ with
$T = ( \{0,1,2,3,4\}, \{ (0,1), (1,2), (1,3), (3,4) \} )$ and $\epsilon_1, \epsilon_2$
map $\vec{P}_1$ to the arcs $(1,2)$ and $(3,4)$ respectively. The proof of Theorem~\ref{thm:tap}
gives an exponential construction of a right adjoint of $\Gamma_{\mathcal{V}}$, while
the composition $\delta_R \circ \Omega_{\T}$ is doubly exponential.

Finally we note that with multiply exponential constructions,
the conditions defining adjacency may become increasingly intricate. 
Consider the path $P=\vec P_2=0\to1\to2$ and let $Q$ be the path $0\to 1\to2\leftarrow 0'\to 1'\to 2'$. Put $\epsilon_1[P]=0\to 1\to 2$ and
$\epsilon_2[P]=0'\to 1'\to 2'$. 
For the Pultr template $\T=(P,Q,\epsilon_1,\epsilon_2)$,
$\Gamma_{\T}$ does admit a right adjoint. We give the following
doubly exponential construction for $\Omega_{\T}$.
\begin{compactitem}
\item The vertices of $\Omega_T(H)$ are quadruples
$R=(R^{--},R^{-+},R^{+-},R^{++})$ such that each of the four sets is
a set of sets of vertices of~$H$ (that is, each $R^{**}\subseteq 2^{V(H)}$) 
and for any $M\in R^{-+}$ and any $N\in R^{+-}$ we have
$M\cap N\ne\emptyset$.
\item There is an arc $R\to S$ in~$\Omega_\T(H)$ if and only if
$\bigcup R^{++}\Rrightarrow\bigcup S^{--}$, $R^{+-}\cap
S^{--}\ne\emptyset$, and $R^{++}\cap S^{-+}\ne\emptyset$.
\end{compactitem}
It can be shown that $\Gamma_\T(G)\to H$ if and only if $G\to\Omega_\T(H)$.
Both conditions ``$A$~intersects~$B$'' and ``every element of~$A$ intersects
every element of~$B$'' are used in the construction of~$\Omega_\T(H)$. 
If the converse of Theorem~\ref{thm:tree} holds, increasingly complex relations 
may be needed to describe right adjoints corresponding to each suitable Pultr template.

%% %%% %%%
\section{Example: A tree template} \label{sec:tree-example}

Our final example models the proof of Theorem~\ref{thm:tap}.

\begin{defi}
\label{defi:tree-tem}
Let $P=\vec P_1=\bigl(\{0,1\},\{(0,1)\}\bigr)$; let $Q$ have vertex set
$\{0,1,\dotsc,10\}$ and arcs $0\to1\to2\to3\to4$, $6\to7\to8\to9\to10$,
$7\to5\to3$. Let $\epsilon_1:P\to Q$, $\epsilon_1(0)=0$, $\epsilon_1(1)=1$, and $\epsilon_2:P\to
Q$, $\epsilon_2(0)=9$, $\epsilon_2(1)=10$.
Consider the Pultr template $\T=(P,Q,\epsilon_1,\epsilon_2)$, see figure.

For a digraph~$H$, define $\Omega_\T(H)$ by:
\begin{multline*}
V(\Omega_\T(H)) = \Bigl\{(S^-,S^+, S^{--}, S^{++}, S^{---}, S^{+++}, S^{-*+++}, S^{--*+++})
	\in (2^{V(H)})^8\colon \\
	\qquad\qquad\text{if } S^+\ne\emptyset \text{, then } S^{---} \Rrightarrow S^{--*+++}
	\Bigr\}
\end{multline*}
and $S\to T$ in $\Omega_\T(H)$ if
\begin{compactitem}
\item $S^+\cap T^- \ne \emptyset$,
\item $S^- \subseteq T^{--}$,
\item $S^{--} \subseteq T^{---}$,
\item $T^{+}\subseteq S^{++}$,
\item $T^{++}\subseteq S^{+++}$,
\item $S^{-*+++}\subseteq T^{--*+++}$, and
\item $S^-=\emptyset$ or $S^{+++}\subseteq T^{-*+++}$.
\end{compactitem}
\end{defi}

\kern-8\baselineskip
\begin{flushright}
\includegraphics{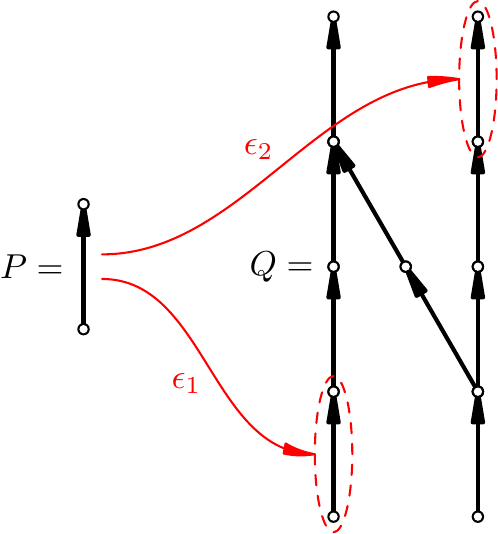}
\end{flushright}

\begin{prop}
Let $\T$ and $\Omega_\T$ be as in Definition~\ref{defi:tree-tem}.
Then for any digraphs $G$, $H$ we have $\Gamma_\T(G)\to H$ if and
only if $G\to\Omega_\T(H)$.
\end{prop}

\begin{proof}
Let $g:G\to \Omega_\T(H)$. Define $f:V(\Gamma_\T(G)) \to V(H)$ by setting $f(u,v)$
to be any element of the (nonempty) set $g(u)^+\cap g(v)^-$. If
$(u,v)\to (x,y)$ in~$\Gamma_\T(G)$, then there exists $h:Q\to G$ such that
$h(0,1)=(u,v)$, $h(9,10)=(x,y)$. By definition, $f(u,v)\in g(v)^-$.
Because $h$ and $g$ are homomorphisms, $g(v)=g(h(1))\to
g(h(2))$ in $\Omega_\T(H)$, thus $g(v)^-\subseteq g(h(2))^{--}$.
Similarly,
\[ f(u,v)\in g(v)^- \subseteq g(h(2))^{--} \subseteq g(h(3))^{---}, \]
and
\[ f(x,y) \in g(x)^+ \subseteq g(h(8))^{++} \subseteq g(h(7))^{+++}
	\subseteq g(h(5))^{-*+++} \subseteq g(h(3))^{--*+++}; \]
here, for the third inclusion we need to observe that $g(h(7))^-\ne\emptyset$,
which follows from the existence of the arc $g(h(6))\to g(h(7))$.
Moreover, $g(h(3))\to g(h(4))$, so $g(h(3))^+\ne\emptyset$,
and hence $g(h(3))^{---}\Rrightarrow g(h(3))^{--*+++}$.
Therefore $f(u,v)\to f(x,y)$ in $\Omega_\T(H)$ and consequently $f$~is a
homomorphism.

For any $u\in V(G)$, let $u^+$ be the set of all arcs outgoing
from~$u$ in~$G$, let $u^{++}$ be the set of all arcs outgoing from
outneighbours of~$u$, etc. In this way, for instance, $u^{+++}$
will be the set of images of the arc $(9,10)$ of~$Q$, under any
homomorphism~$h$ of $Q[7,8,9,10]$, the subtree of~$Q$ on
$\{7,8,9,10\}$, to~$G$, with the property that $h(7)=u$.
Analogously, let
\begin{align*}
u^{-*+++} &= \{ h(9,10)\colon h:Q[5,6,\dotsc,10]\to G,\ h(5)=u \}, \\
u^{--*+++} &= \{ h(9,10)\colon h:Q[4,5,\dotsc,10]\to G,\ h(4)=u \}.
\end{align*}

Now let $f:\Gamma_\T(G)\to H$. Define $g:V(G)\to V(\Omega_\T(H))$ by \[g(u)=(f[u^-],
f[u^+], f[u^{--}], \dotsc, f[u^{--*+++}]).\] If $u^+\ne\emptyset$,
$a_1\in u^{---}$, $a_2\in u^{--*+++}$, then we observe that there
is a homomorphism $h:Q\to G$ such that $h(3)=u$, $h(0,1)=a_1$,
$h(3,4)\in u^+$, and $h(9,10)=a_2$. Thus $a_1\to a_2$ in~$\Gamma_\T(G)$,
and hence $f(a_1)\to f(a_2)$ in~$H$. Therefore whenever
$g(u)^+\ne\emptyset$, we have $g(u)^{---}\Rrightarrow g(u)^{--*+++}$,
so each $g(u)$ is indeed a vertex of~$\Omega_\T(H)$.

If $u\to v$ in~$G$, then $f(u,v)\in g(u)^+\cap g(v)^-\ne\emptyset$.
Checking the inclusions in the definition of an arc of~$\Omega_\T(H)$ is
then relatively easy. We can conclude that $g(u)\to g(v)$ in~$\Omega_\T(H)$,
so $g$~is a homomorphism.
\end{proof}

\section{Right Pultr adjoints for tree templates}
\label{sec:tree}

In this section we consider templates $\T=(P,Q,\epsilon_1,\epsilon_2)$,
where $P=\vec P_0$ (a single vertex) or $P=\vec P_1$ (a single arc),
and $Q$~is a tree. We prove the following:
\begin{theorem} \label{thm:tap}
Let $\T=(P,Q,\epsilon_1,\epsilon_2)$ be a Pultr template such that
$P=\vec P_0$ or $P=\vec P_1$ and $Q$~is a tree. Then there exists
a functor~$\Omega_\T$ such that for any digraphs $G,H$ we have
$\Gamma_\T(G)\to H$ if and only if $G\to\Omega_\T(H)$.
\end{theorem}
As a matter of fact, many different non-isomorphic constructions
are possible (and we shall hint at some variations here as well),
but they are all homomorphically equivalent.

\begin{proof}[Proof of Theorem~\ref{thm:tap}]
There is a pathologically trivial case where $\epsilon_1=\epsilon_2$, which we
do not consider in the following exposition.

\paragraph{The subtrees of $Q$.}

Every vertex of $\Rho_\T(H)$ will be a vector of subsets of~$V(H)$,
indexed by some rooted subtrees of~$Q$.  Namely, the rooted subtrees
are determined as follows:

The images $\epsilon_1[P]$ and $\epsilon_2[P]$ in~$Q$ are connected
by a path; call this path~$\tilde Q$. The path~$\tilde Q$ contains
one vertex from each of $\epsilon_1[P]$ and $\epsilon_2[P]$. Thus in
the template of Definition~\ref{defi:path}, $\tilde Q=Q$; in the arc graph construction,
$\tilde Q$~is just the vertex~$1$; whereas for the $\T$ of Section~\ref{sec:tree-example},
$\tilde Q$~is the path $1\to 2\to 3\leftarrow 5\leftarrow 7\to 8\to 9$.
We choose any vertex $m$ on the path~$\tilde Q$ and call it the
\emph{middle vertex}. (The discretion in our choice of the middle vertex 
is one of the causes for the existence of many different right adjoints.)

If $P=\vec P_0$ and $\epsilon_1[P]$ and $\epsilon_2[P]$ are joined
by an arc, either $\epsilon_1[P]$ or $\epsilon_2[P]$ can be taken
to be the middle vertex. If $P=\vec P_1$ and the images $\epsilon_1[P]$
and $\epsilon_2[P]$ intersect (like in the arc graph template), the
middle vertex is the vertex they intersect in.

Now, every non-leaf~$u$ of~$Q$ is a cut vertex. Consider each subtree
induced by the vertex~$u$ and all the vertices of some component
of ${Q-u}$ (obtained from~$Q$ by removing the vertex~$u$); take
only those components that do not contain the middle vertex. Let
$\S_u$ be the set of all such subtrees, each rooted in~$u$. Finally, let
$\S'=\bigcup\{\S_u\colon u\text{ is a non-leaf of }Q\}$.
Notice that the root of each $T\in\S'$ has degree~$1$ in~$T$.

Some of the trees in~$\S'$ contain either $\epsilon_1[P]$ or
$\epsilon_2[P]$; call this image the \emph{$P$-arc} or \emph{$P$-vertex}
of that tree (depending on whether $P=\vec P_1$ or $P=\vec P_0$).
Trees from~$\S'$ with no $P$-arc or $P$-vertex are called \emph{pendent
subtrees} because they appear to hang from the path~$\tilde Q$.
Each $\S_u$ contains at most one non-pendent tree~$T_u$ unless
$u$~is the middle vertex~$m$; then $\S_m$~may contain two non-pendent
subtrees: $T_{m,1}$ containing $\epsilon_1[P]$ and $T_{m,2}$ containing~$\epsilon_2[P]$.

Finally, let $\S$ consist of the trees in~$\S'$, taken up to isomorphism.
An isomorphism is meant to preserve not only the
vertices and arcs, but the root and the $P$-vertex or $P$-arc as well.%
\footnote{In fact, we may take these trees up to homomorphic equivalence,
with homomorphisms preserving arcs, the root and the $P$-vertex or $P$-arc.}

\paragraph{The vertices of $\Rho_\T(H)$.}

A vertex of $\Rho_\T(H)$ is
\begin{compactitem}
\item any vector $(R\bul,\ R^T\colon T\in\S)$ if $P=\vec P_0$,
\item any vector $(R^T\colon T\in\S)$ if $P=\vec P_1$,
\end{compactitem}
where $R\bul\in V(H)$,
$R^T\in\{0,1\}$ if $T$~is a pendent subtree, and $R^T\subseteq V(H)$
otherwise, if it satisfies the following condition (remember that $m$~is the middle vertex of~$Q$):
\begin{equation}
\label{eq:vert}
\text{If } R^T=1 \text{ for every pendent } T\in\S_m \text{, then }
R^{T_{m,1}} \Rrightarrow R^{T_{m,2}}.
\end{equation}
Recall that $A\Rrightarrow B$ means that $a\to b$ for any vertex $a\in A$
and any vertex $b\in B$; here the arc is meant to exist in~$H$.

\paragraph{The arcs of $\Rho_\T(H)$.}
Let $R,S$ be vertices of $\Rho_\T(H)$. Thus $R=(R\bul,R^T\colon {T\in\S})$,
$S=(S\bul,S^T\colon {T\in\S})$, or $R=(R^T\colon T\in\S)$, $S=(S^T\colon
T\in\S)$. Then $R\to S$ in $\Rho_\T(H)$ if and only if all of the
following conditions are satisfied:

\begin{itemize}

\item If $P=\vec P_0$:
For an arc~$e$ of the path~$\tilde Q$ with vertices $a,b$ such that
$a=\epsilon_i[P]$ for some $i\in\{1,2\}$, let $T_b$ be the non-pendent
tree rooted in~$b$ that contains~$a$.
If $e=(a,b)$, then we have the condition
\begin{equation}
\label{eq:arc-1}
\text{if $R^T=1$ for every pendent $T\in\S_a$, then } R\bul \in S^{T_b};
\end{equation}
if $e=(b,a)$, then we have the condition
\begin{equation}
\label{eq:arc-2}
\text{if $S^T=1$ for every pendent $T\in\S_a$, then } S\bul \in R^{T_b}.
\end{equation}

\item If $P=\vec P_1$: Let $\epsilon_1[P]$ consist of vertices
$a,b$, and let $\epsilon_2[P]$ consist of vertices $c,d$, so that
$b$ and $c$ would be ``closer'' to the middle vertex~$m$; that is,
$b,c$~are vertices of the path~$\tilde Q$. (If the images $\epsilon_1[P]$
and $\epsilon_2[P]$ intersect, then $b=c=m$.) The sets $\S_a$ and
$\S_d$ contain only pendent trees. Let $T_b$ be the non-pendent
tree in~$\S_b$ that contains~$a,b$, and let $T_c$ be the non-pendent
tree in~$\S_c$ that contains~$c,d$. If $a\to b$ in~$Q$, put $A=R$
and $B=S$; if on the other hand $b\to a$ in~$Q$, put $A=S$ and
$B=R$. Analogously, if $c\to d$ in~$Q$, put $C=R$ and $D=S$; if
$d\to c$ in~$Q$, put $C=S$ and $D=R$. Then we have the following
three conditions:

\begin{align}
\label{eq:arc-5}
&\text{If $A^T=1$ for every $T\in\S_a$, then $B^{T_b}\ne\emptyset$.}\\
\label{eq:arc-6}
&\text{If $D^T=1$ for every $T\in\S_d$, then $C^{T_c}\ne\emptyset$.}\\
\label{eq:arc-7}
&\text{If $A^T=1$ for every $T\in\S_a$ and $D^T=1$ for every $T\in\S_d$, then $B^{T_b}\cap C^{T_c}\ne\emptyset$.}
\end{align}

\item For any arc~$e$ of the path~$\tilde Q$ not covered by conditions
\eqref{eq:arc-1}--\eqref{eq:arc-2}, let $a,b$ be the vertices of~$e$
so that every tree in $\S_a$ is a subtree of~$T_b$, a non-pendent
tree in~$\S_b$.
If $e=(a,b)$, then we have the condition
\begin{equation}
\label{eq:arc-3}
\text{if $R^T=1$ for every pendent $T\in\S_a$, then } R^{T_a} \subseteq S^{T_b};
\end{equation}
if $e=(b,a)$, then we have the condition
\begin{equation}
\label{eq:arc-4}
\text{if $S^T=1$ for every pendent $T\in\S_a$, then } S^{T_a} \subseteq R^{T_b}.
\end{equation}

\item For any other arc~$e$ of~$Q$, let $a,b$ be the vertices of~$e$ and let $T'\in\S_b$ so that
every tree in $\S_a$ is a subtree of~$T'$.
If $e=(a,b)$, then we have the condition
\begin{equation}
\label{eq:arc-8}
\text{if $R^T=1$ for every $T\in\S_a$, then } S^{T'}=1;
\end{equation}
if $e=(b,a)$, then we have the condition
\begin{equation}
\label{eq:arc-last}
\text{if $S^T=1$ for every $T\in\S_a$, then } R^{T'}=1.
\end{equation}
\end{itemize}

Thus we might say that it is really tough to be an arc.

Please take another look at the example in Section~\ref{sec:tree-example}.
Which vertex of~$Q$ did we choose to be the middle vertex?  You
should be warned that in the example, we combined the roles of the
non-pendent tree and the pendent tree~$\vec P_1$ rooted in the
tail, as well as the roles of the non-pendent tree and the
pendent tree~$\vec P_1$ rooted in the head. This has led to a
construction not isomorphic but homomorphically equivalent to the
recipe given here.

\paragraph{The homomorphisms.}

Let $f:\Gamma_\T(G) \to H$ be a homomorphism.  We will define a
mapping $g:V(G)\to V(\Omega_\T(H))$.  For a vertex~$u$ of~$G$ and
a subtree $T\in\S$ rooted in~$t$, define $g(u)^T$ in the following
way:

\noindent
If $T$ is pendent, let
\begin{equation}
\label{eq:etautpen}
g(u)^T =
\begin{cases}
1 & \text{if there exists a homomorphism $h:T\to G$ such that } h(t) = u,\\
0 & \text{otherwise}.
\end{cases}
\end{equation}
If $T$ is non-pendent, let $p$ be the $P$-vertex or the $P$-arc of~$T$ and let
\begin{equation}
\label{eq:etautnonpen}
g(u)^T = \bigl\{\, f(h(p))\colon h:T\to G \text{ is a homomorphism such that } h(t)=u \,\bigr\}.
\end{equation}
Moreover, if $P=\vec P_0$, set
\begin{equation}
g(u)\bul = f(u).
\end{equation}

First we need to verify that $g(u)$ is indeed a vertex of~$\Omega_\T(H)$.
Thus we need to verify the condition~\eqref{eq:vert}.  Suppose that
$u\in V(G)$ and that $g(u)^T=1$ for every pendent $T\in\S_m$. Hence
by definition there exist homomorphisms $h_T:T\to G$ with $h_T(m)=u$.
Let $x_1\in g(u)^{T_{m,1}}$ and $x_2\in g(u)^{T_{m,2}}$. Then
by~\eqref{eq:etautnonpen} there are homomorphisms $h_i:T_{m,i}\to
G$ with $h_i(m)=u$ and $x_i=f(h(\epsilon_i[P]))$ for $i=1,2$. Since
all the trees in~$\S_m$ share only the vertex~$m$, we can define a
homomorphism $h:Q\to G$ by putting $h(m)=u$, $h(a)=h_i(a)$ if $a$~is
a vertex of a non-pendent~$T_{m,i}$, and $h(a)=h_T(a)$ if $a$~is a
vertex of a pendent $T\in\S_m$. Thus by definition $h(\epsilon_1[P])\to
h(\epsilon_2[P])$ in~$\Gamma_\T(G)$. Because $f$~is a homomorphism,
$x_1=f(h(\epsilon_1[P]))\to f(h(\epsilon_2[P]))=x_2$. Therefore
$g(u)^{T_{m,1}}\Rrightarrow g(u)^{T_{m,2}}$ as we were supposed to
prove.

We aim to show that $g$~is a homomorphism, that is, that whenever
$u\to v$ in~$G$, then $g(u)\to g(v)$ in~$\Omega_\T(H)$. Thus we
need to check all the conditions \eqref{eq:arc-1}--\eqref{eq:arc-last}.
The conditions are all similar in nature and they are all proved
by stitching together homomorphisms in a fashion similar to
that of the previous paragraph.

Let us start with condition~\eqref{eq:arc-1}. Suppose that $P=\vec
P_0$, $e=(a,b)$ is an arc of~$\tilde Q$, $a=\epsilon_i[P]$. Let
$T_b$ be the non-pendent tree rooted in~$b$ that contains~$a$. If
$g(u)^T=1$ for every pendent $T\in\S_a$, then by~\eqref{eq:etautpen}
there exist homomorphisms $h_T:T\to G$ with $h_T(a)=u$. Since
$T_b$~consists of the union of the trees~$T\in\S_a$ and the
arc~$(a,b)$, we may define $h:T_b\to G$ by putting $h(a)=u$, $h(b)=v$,
and $h(x)=h_T(x)$ for any other vertex~$x$ of~$T_b$, where $T$~is
the unique tree $T\in\S_a$ that contains~$x$. By our assumption
$u\to v$ in~$G$, so $h$~is a homomorphism. Then by~\eqref{eq:etautnonpen}
we have $g(u)\bul=f(u)=f(h(a))\in g(v)^{T_b}$, which
verifies~\eqref{eq:arc-1}.

Condition \eqref{eq:arc-2} is analogous to~\eqref{eq:arc-1}.

Suppose now that $P=\vec P_1$, $\epsilon_1[P]=(a,b)$,
$\epsilon_2[P]=(c,d)$, $b,c$~are vertices of the path~$\tilde Q$.
Let $T_b$ be the non-pendent tree rooted in~$b$ that contains~$a$
and let $T_c$ be the non-pendent tree rooted in~$c$ that contains~$d$.
If $g(u)^T=1$ for every $T\in\S_a$, then by~\eqref{eq:etautpen} there
exist homomorphisms $h_T:T\to G$ with $h_T(a)=u$. As above, we can
define  a homomorphism $h:T_b\to G$ by putting $h(a)=u$, $h(b)=v$,
and $h(x)=h_T(x)$ with $T$ being the corresponding tree in~$\S_a$.
Since $(a,b)$~is the $P$-arc of~$T_b$, by~\eqref{eq:etautnonpen}
we have $f(u,v)\in g(v)^{T_b}$; hence $g(v)^{T_b}\ne\emptyset$,
which verifies~\eqref{eq:arc-5}. By an analogous argument we can
show that if $g(v)^T=1$ for every $T\in\S_d$, then $f(u,v)\in
g(u)^{T_c}\ne\emptyset$, which verifies~\eqref{eq:arc-6}. Therefore
if both $g(u)^T=1$ for every $T\in\S_a$ and $g(v)^T=1$ for every
$T\in\S_d$, then $f(u,v)\in g(v)^{T_b}\cap g(u)^{T_c}\ne\emptyset$,
which verifies condition~\eqref{eq:arc-7}. If $\epsilon_1[P]=(b,a)$
and/or $\epsilon_2[P]=(d,c)$, the proof is analogous.

For condition~\eqref{eq:arc-3}, let $e=(a,b)$ be an arc of~$Q$, let
$T_b$ be the non-pendent tree rooted in~$b$ that contains~$a$, let
$T_a$ be the non-pendent tree rooted in~$a$, and suppose that
$g(u)^T=1$ for every pendent $T\in\S_a$; thus by~\eqref{eq:etautpen}
for every such~$T$ there exists a homomorphism $h_T:T\to G$ such
that $h_T(a)=u$. Whenever $x\in g(u)^{T_a}$, by~\eqref{eq:etautnonpen}
there exists a homomorphism $h_{T_a}:T_a\to G$ such that $h(a)=u$
and $x=f(h_{T_a}(p))$, where $p$~is the $P$-vertex or the $P$-arc
of~$T_a$ (thus also the $P$-vertex or the $P$-arc of~$T_b$). Note
that $T_b$~is the union of the arc~$(a,b)$ and of all the trees
in~$\S_a$; so there is a homomorphism $h:T_b\to G$ such that $h(b)=v$
that coincides with~$h_T$ on each subtree~$T\in\S_a$, in particular,
$h(p)=x$.  Hence, by~\eqref{eq:etautnonpen}, $x\in g(v)^{T_b}$.
Therefore $g(u)^{T_a}\subseteq g(v)^{T_b}$ as we are supposed to
show.

Condition~\eqref{eq:arc-4} is analogous to condition~\eqref{eq:arc-3},
and the remaining two conditions are very similar as well.
Thus we have shown that if $\Gamma_\T(G)\to H$, then
$G\to\Omega_\T(H)$.

To show the opposite implication, suppose that $g:G \to \Rho_\T(H)$
is a homomorphism. First, if $P=\vec P_0$, then vertices of~$\Gamma_\T(G)$
are essentially the same as vertices of~$G$. For $u\in V(G)$, put
\begin{equation}
f(u) = g(u)\bul.
\end{equation}

If on the other hand $P=\vec P_1$, then vertices of~$\Gamma_\T(G)$
are essentially the same as arcs of~$G$.  If $(u,v)$ is an arc
of~$G$, then $g(u)\to g(v)$ in~$\Rho_\T(H)$.  Let $a,b,c,d$, $T_b,T_c$
and $A,B,C,D$ be as in conditions \eqref{eq:arc-5}--\eqref{eq:arc-6}
with $R=g(u)$, $S=g(v)$. If the hypothesis of condition~\eqref{eq:arc-7}
is satisfied, define $f(u,v)$ to be an arbitrary vertex in~$B^{T_b}\cap
C^{T_c}$. If only the hypothesis of condition~\eqref{eq:arc-5} is
satisfied, define $f(u,v)$ to be an arbitrary vertex in~$B^{T_b}$;
if only the hypothesis of condition~\eqref{eq:arc-6} is satisfied,
define $f(u,v)$ to be an arbitrary vertex in~$C^{T_c}$. Otherwise
let $f(u,v)$ be an arbitrary vertex of~$H$. (Note that if $H$~has
no vertices, then $\Omega_\T(H)$~has no arcs; thus if $G\to\Omega_\T(H)$,
then $\vec P_1\notto G$, and so $\Gamma_\T(G)$~has no vertices.)

\begin{claim}
\label{calim1}
For any pendent $T'\in\S$ rooted in some~$b$ and any homomorphism
$h:T'\to G$ we have $g(h(b))^{T'}=1$.
\end{claim}

\begin{proof}
By induction on the number of arcs of~$T'$. If $T'$ has one arc,
without loss of generality we may assume it is $a\to b$. Then
$\S_a=\emptyset$ and we have $g(h(a))\to g(h(b))$ in~$\Omega_\T(H)$;
hence by~\eqref{eq:arc-8}, $g(h(b))^{T'}=1$. (If $a\leftarrow b$,
then condition~\eqref{eq:arc-last} applies.)

If $T'$ has more than one arc, then still the root~$b$ has degree~$1$
in~$T'$, so there is a unique arc (w.l.o.g.) $a\to b$. By induction,
for every $T\in\S_a$ we have $g(h(a))^{T}=1$. Since $g(h(a))\to
g(h(b))$ in~$\Omega_\T(H)$, by~\eqref{eq:arc-8} we have $g(h(b))^{T'}=1$.
(Again, if $a\leftarrow b$, then condition~\eqref{eq:arc-last} applies.)
\end{proof}

\begin{claim}
\label{calim2}
Let $b$ be any vertex of the path~$\tilde Q$ in~$Q$.
Let $T_b$ be a non-pendent tree rooted in~$b$ and let $i\in\{1,2\}$
be such that $T_b$~contains~$p_i=\epsilon_i[P]$.
Let $h:T_b\to G$ be a homomorphism.
Then $f(h(p_i))\in g(h(b))^{T_b}$.
\end{claim}

\begin{proof}
The proof is by induction on the distance of~$b$ from the $P$-arc
or $P$-vertex of~$Q$. First, if $P=\vec P_0$ and there is an arc
$p_i\to b$ or $p_i\leftarrow b$, then by Claim~\ref{calim1} applied
to all the pendent trees in~$\S_{p_i}$ and the corresponding
restrictions of~$h$ we get $g(h(p_i))^{T}=1$ for every $T\in\S_{p_i}$.
If $p_i\to b$ in~$T_b$, then $h(p_i)\to h(b)$ in~$G$ and $g(h(p_i))\to
g(h(b))$ in~$\Omega_\T(H)$. Hence $f(h(p_i))=g(h(p_i))\bul\in
g(h(b))^{T_b}$ by~\eqref{eq:arc-1}. If, on the other hand,
$p_i\leftarrow b$ in~$T_b$, then condition~\eqref{eq:arc-2} applies
analogously.

Next, let $P=\vec P_1$ and let $p_i=(a,b)$. The non-pendent tree~$T_b$
consists of the arc $(a,b)$ and all the pendent trees in~$\S_a$.
Again, by Claim~\ref{calim1} we have $g(h(a))^{T}=1$ for every
pendent $T\in\S_a$. Thus by the definition of~$f$ we have
$f(h(p_i))=f(h(a,b))\in g(h(b))^{T_b}$. The case $p_i=(b,a)$ is
analogous.

Finally, let $T_b$ contain the arc~$(a,b)$ and suppose that $a$~is
not the $P$-vertex of~$T_b$ and $(a,b)$~is not the $P$-arc of~$T_b$.
Let $T_a$ be the non-pendent subtree of~$T_b$ rooted in~$a$. By
induction, $f(h(p_i))\in g(h(a))^{T_a}$. For every pendent $T\in\S_a$,
an application of Claim~\ref{calim1} to the restriction of~$h$
to~$T$ shows that $g(h(a))^T=1$. As $a\to b$ in~$\T_b$, we have
$h(a)\to h(b)$ in~$G$ and $g(h(a))\to g(h(b))$ in~$\Omega_\T(H)$.
Hence $g(h(a))^{T_a}\subseteq g(h(b))^{T_b}$ by~\eqref{eq:arc-3}.
Therefore $f(h(p_i))\in g(h(b))^{T_b}$. Just like before, the case
$b\to a$ is analogous; condition~\eqref{eq:arc-4} applies.
\end{proof}

We aim to show that $f$~is a homomorphism. Let $y\to z$ in
$\Gamma_\T(G)$; $y$~and~$z$ may be vertices or arcs of~$G$, depending
on whether $P=\vec P_0$ or $P=\vec P_1$. Then there exists a
homomorphism $h:Q\to G$ such that $y=(h\circ \epsilon_1)[P]$ and
$z=(h\circ \epsilon_2)[P]$.  Put $p_1=\epsilon_1[P]$ and
$p_2=\epsilon_2[P]$, so that $y=h(p_1)$ and $z=h(p_2)$.

For $i=1,2$, let $h_i$ be the restriction of~$h$ to the non-pendent
tree~$T_{m,i}$ rooted in the middle vertex~$m$ of~$Q$.  By
Claim~\ref{calim2}, $f(y)=f(h_1(p_1))\in g(h_1(m))^{T_{m,1}}=
g(h(m))^{T_{m,1}}$ and $f(z)=f(h_2(p_2))\in g(h_2(m))^{T_{m,2}}=
g(h(m))^{T_{m,2}}$. By~\eqref{eq:vert}, $g(h(m))^{T_{m,1}} \Rrightarrow
g(h(m))^{T_{m,1}}$. Therefore $f(y)\to f(z)$ in~$H$. Indeed, $f$~is
a homomorphism of~$\Gamma_\T(G)$ to~$H$. 
This concludes the proof of Theorem~\ref{thm:tap}.
\end{proof}

\section{Pultr functors and homomorphism dualities}

In our final section, we return to the connection between Pultr
functors and homomorphism dualities, which enabled us to prove
Theorem~\ref{thm:tree}.

\begin{defi}
Let $\F$ be a set of digraphs and $H$ a digraph. We say that $\F$~is
a \emph{complete set of obstructions} for~$H$ or that $(\F,H)$ is
a \emph{homomorphism duality} if for any digraph~$G$,
\[ G\to H \quad \Leftrightarrow \quad \forall\, F\in\F\colon F\notto G. \]
We also say that $H$ has \emph{tree duality} if it admits a complete
set of obstructions all of whose elements are trees, and $H$~has
\emph{finite duality} if it admits a finite complete set of
obstructions.
\end{defi}

Connections between left and central Pultr functors and homomorphism
dualities were the topic of our paper~\cite{FonTar:Adjoint}. Here
we consider their relationship to the right adjoints.

\begin{theorem}
Let $\T$ be a Pultr template. Suppose that $(\F,H)$ is a homomorphism
duality.

\begin{enumerate}[(1)]

\item If there exists a digraph~$K$ such that
\begin{equation}
\label{eq:K}
\text{for any digraph $G$,} \qquad
\Gamma_\T(G) \to H \quad \Leftrightarrow \quad G \to K,
\end{equation}
then $\{\Lambda_\T(F)\colon F\in\F\}$ is a complete set of obstructions
for~$K$.
\end{enumerate}

Assume, moreover, that the template $\T$ satisfies the necessary
conditions for the existence of~$\Omega_\T$ given by
Theorem~\ref{thm:tree}.

\begin{enumerate}[(1)]\setcounter{enumi}{1}
\item If $H$ has tree duality and $K$ satisfies~\eqref{eq:K}, then
$K$ also has tree duality.

\item If $H$ has finite duality, then there does exist a digraph~$K$
that satisfies~\eqref{eq:K}. Moreover, $K$~also has finite duality.
\end{enumerate}
\end{theorem}

\begin{proof}
(1) For any digraph $G$,
\begin{align*}
G\to K & \Leftrightarrow \Gamma_\T(G) \to H \\
& \Leftrightarrow \forall\, F\in\F\colon F\notto \Gamma_\T(G) \\
& \Leftrightarrow \forall\, F\in\F\colon \Lambda_\T(F)\notto G .
\end{align*}

(2) Let $\F$ be a complete set of obstructions for~$H$ consisting
entirely of trees. Then $\Lambda_\T(F)$~is homomorphically equivalent
to a tree~$\Lambda'(F)$ for each $F\in\F$. Since for any digraph~$G$
we have $\Lambda_\T(F)\to G$ iff $\Lambda'(F)\to G$, the set
$\F'=\{\Lambda'(F)\colon F\in\F\}$ is a complete set of tree obstructions
for~$K$.

(3) Let $\F'=\{\Lambda'(F)\colon F\in\F\}$ as above. By~\cite{NesTar:Dual},
there exists a digraph~$K$ such that $(\F',K)$ is a homomorphism
duality. By the above equivalence, $K$~satisfies~\eqref{eq:K}.
\end{proof}

Thus for {\em all} the templates $\T$ satisfying the necessary
hypotheses in Theorem~\ref{thm:tree}, the central Pultr functor
$\Gamma_{\T}$ admits a partial right adjoint~$\Omega_{\T}$, defined
at least on some subclass~$\mathcal{D}$ of the class of all digraphs:
namely, we can take the class of all digraphs with finite duality.

Furthermore, combining Theorem~\ref{thm:tap}, compositions and
sporadic examples as in Section~\ref{sec:compo}, we get a wide class
of templates satisfying the necessary hypotheses in Theorem~\ref{thm:tree}
for which the central Pultr functor admits a right adjoint.  This
gives reason to think that the converse of Theorem~\ref{thm:tree}
might hold. While known constructions of finite duals of trees are
rather complex, a right adjoint of a central Pultr functor is even
more general. This fact justifies the complexity of the construction
given in Section~\ref{sec:tree}.

%%%%%%%%%%%%%%%%%%%%%%%%%%%%%%%%% biblio hacks


\begin{thebibliography}{1}

\bibitem{FonTar:Adjoint}
J.~Foniok and C.~Tardif.
\newblock Adjoint functors and tree duality.
\newblock {\em Discrete Math. Theor. Comput. Sci.}, 11(2):97--110, 2009.

\bibitem{HaT:Graph-powers}
H.~Hajiabolhassan and A.~Taherkhani.
\newblock Graph powers and graph homomorphisms.
\newblock {\em Electron. J. Combin.}, 17:R17, 16 pages, 2010.

\bibitem{HelNes:GrH}
P.~Hell and J.~Ne{\v s}et{\v r}il.
\newblock {\em Graphs and Homomorphisms}, volume~28 of {\em Oxford Lecture
  Series in Mathematics and Its Applications}.
\newblock Oxford University Press, 2004.

\bibitem{Kom:Phd}
P.~Kom{\'a}rek.
\newblock {\em Good characterisations in the class of oriented graphs}.
\newblock PhD thesis, Czechoslovak Academy of Sciences, Prague, 1987.
\newblock In Czech (Dobr{\'e} charakteristiky ve t{\v r}{\'\i}d{\v e}
  orientovan{\'y}ch graf{\r u}).

\bibitem{NesTar:Dual}
J.~Ne{\v s}et{\v r}il and C.~Tardif.
\newblock Duality theorems for finite structures (characterising gaps and good
  characterisations).
\newblock {\em J. Combin. Theory Ser. B}, 80(1):80--97, 2000.

\bibitem{Pul:The-right-adjoints}
A.~Pultr.
\newblock The right adjoints into the categories of relational systems.
\newblock In {\em Reports of the Midwest Category Seminar, IV}, volume 137 of
  {\em Lecture Notes in Mathematics}, pages 100--113, Berlin, 1970. Springer.

\bibitem{Tar:Mul}
C.~Tardif.
\newblock Multiplicative graphs and semi-lattice endomorphisms in the category
  of graphs.
\newblock {\em J. Combin. Theory Ser. B}, 95(2):338--345, 2005.

\end{thebibliography}
\end{document}